\documentclass[a4paper,11pt,fullpage]{article}
\usepackage[english]{babel}

\usepackage[left=1in,right=1in,top=1in,bottom=1in]{geometry}
\usepackage{graphicx}
\usepackage[dvipsnames]{xcolor}%
\usepackage{amsmath}
\usepackage{amsthm}
\usepackage{amsfonts}
\usepackage{amssymb}
\usepackage{amsxtra}
\usepackage{latexsym}
\usepackage{ifthen}
\usepackage{cite}
\usepackage{esint}
\usepackage[cp1250]{inputenc}
\usepackage{ulem}

\usepackage{epic}
\usepackage{eepic}
\usepackage{indentfirst}

\DeclareMathOperator*{\osc}{osc}

\numberwithin{equation}{section}
\newtheorem{theorem}{Theorem}[section]
\newtheorem{lemma}{Lemma}[section]
\newtheorem{remark}{Remark}[section]
\newtheorem{definition}{Definition}[section]

\newtheorem{propo}{Proposition}[section]
\newtheorem*{claim}{Claim}

\def\XXint#1#2#3{{\setbox0=\hbox{$#1{#2#3}{\int}$}
     \vcenter{\hbox{$#2#3$}}\kern-.5\wd0}}

\begin{document}

\title{ On the  continuity of solutions to the anisotropic $N$-Laplacian with $L^1$ lower order term}
\author{ Mariia O. Savchenko, Igor I. Skrypnik, Yevgeniia A. Yevgenieva
 }

\date{}
\maketitle

\begin{abstract}
We establish the continuity of bounded solutions to the anisotropic elliptic equation
$$-\sum\limits_{i=1}^N\Big(|u_{x_i}|^{p_i-2} u_{x_i}\Big)_{x_i}=f(x),\quad x\in \Omega,\quad f(x)\in L^1(\Omega)	$$
under the conditions
$$\min\limits_{1\leqslant i\leqslant N} p_i >1,\quad \sum\limits_{i=1}^N \frac{1}{p_i}=1$$
and
$$\lim\limits_{\rho\rightarrow 0}\,\sup\limits_{x\in \Omega}\int\limits^{\rho}_0\Big(\int\limits_{B_r(x)}|f(y)|\,dy\Big)^{\frac{1}{N-1}}\frac{dr}{r}=0.$$
In the standard case $p_1=...=p_N=N$, these conditions recover the known results for the $N$-Laplacian.

%\textbf{Keywords:} continuity of solutions, anisotropic $N$- Laplacian, $L^1$ data.


\end{abstract}

\pagestyle{myheadings} 
\thispagestyle{plain}
\markboth{Maria O. Savchenko, Igor I. Skrypnik, Yevgeniia A. Yevgenieva}
{}

\section{Introduction and main results}\label{Introduction}
For an open bounded set $\Omega\subset \mathbb{R}^N$,  $N\in\mathbb{N},\ N\geqslant2$ we consider the anisotropic elliptic equation
\begin{equation}\label{eq1.1}
-\sum\limits_{i=1}^N (a_i(x, \nabla u))_{x_i}=f(x),\quad x\in \Omega,\quad f(x)\in L^1(\Omega),
\end{equation}
where $a_i : \Omega\times \mathbb{R}^N \to \mathbb{R}$ are measurable functions satisfying the conditions
\begin{equation}\label{eq1.2}
\begin{cases}
        \sum\limits_{i=1}^N  a_i(x, \xi)\, \xi_i \geqslant K_1 \sum\limits_{i=1}^N |\xi_i |^{p_i},\quad 1<p_1\leqslant p_2\leqslant...\leqslant p_N, \quad \sum\limits_{i=1}^N \frac{1}{p_i}=1,\\
        |a_i(x, \xi)|\leqslant K_2 \Big(\sum\limits_{j=1}^N\,|\xi_j|^{p_j}\Big)^{1-\frac{1}{p_i}},\quad i=1,...,N,
    \end{cases}
\end{equation}
where $K_1$ and $K_2$ are positive constants, and $f(x)$ satisfies conditions that will be specified below.

The local boundedness,  H\"{o}lder continuity, and Harnack's inequality of weak solutions to homogeneous linear divergence-type second-order elliptic equations with measurable coefficients, but without lower order terms have been well established since the famous works of De Giorgi \cite{DeG},  Nash \cite{Nas}, and  Moser \cite{Mos1, Mos2}. Later, Ladyzhenskaya and Ural'tseva \cite{LadUra}, as well as Serrin \cite{Ser}, generalized the results of De Giorgi and Moser to the case
of quasi-linear equations of the $p$-Laplace type, incorporating lower order terms under conditions expressed in terms of $L^{t}$-spaces for $t> \dfrac{N}{p}$.
Using probabilistic techniques, Aizenman and Simon, in their renowned paper \cite{AizSim}, proved the Harnack's inequality and the continuity of weak solutions to the equation $\triangle u = V(x) u $ under the local Kato class condition on the potential $V(x)$:
\begin{equation*}
\lim\limits_{R\rightarrow 0}\,\sup\limits_{x\in \Omega}\int\limits_{B_R(x)}\frac{V(y)}{|x-y|^{N-2}}\,dy=0,\quad N \geqslant 3.
\end{equation*}
Moreover, they demonstrated that the Kato-type condition on the potential $V$ is necessary for the validity of the Harnack's inequality. Shortly thereafter, Chiarenza, Fabes, and Garofalo \cite{ChiFabGar}, as well as Kurata \cite{Kur}, developed real-variables techniques to establish the Harnack's inequality for a linear divergence-type equation with measurable coefficients and a potential from the Kato class, thereby extending the results of Aizenman and Simon.
A significant breakthrough in the nonlinear theory was achieved by Kilpel\"{a}inen and Mal\'y \cite{KilMal}, who proved the following point-wise estimate for non-negative solutions of the equation $-\triangle_p u=f(x)\geqslant 0$ in terms of the Wolff potential  $\mathcal{W}^f_{1, p}(x_0, \rho)$:
\begin{equation}\label{eq1.3}
u(x_0)\leqslant \gamma(N, p) \Big(\fint\limits_{B_{\rho}(x_0)} u^{(p-1)(1+\lambda)}\,dx\Big)^{\frac{1}{(p-1)(1+\lambda)}}+\gamma(N, p)\, \mathcal{W}^f_{1,p}(x_0, 2 \rho),
\end{equation}
where $\lambda \in \big(0, \dfrac{p-1}{N-p+1}\big)$, $p\in (1, N]$, and the Wolff potential is given by
\begin{equation}\label{eq1.4}
\mathcal{W}^f_{\alpha, p}(x_0, \rho):=\int\limits_0^{\rho}\Big(r^{\alpha p-N}\int\limits_{B_r(x_0)} f(z) dz\Big)^{\frac{1}{p-1}}\,\frac{dr}{r},\quad 
0< \alpha p\leqslant N.
\end{equation}
Estimate \eqref{eq1.3}, together with the lower bound established by Rakotoson and Ziemer \cite{RakZie}
\begin{equation*}
u(x_0)\geqslant \gamma(N, p) \mathcal{W}^f_{1, p}(x_0, \rho),
\end{equation*}
provides a precise local behavior of solutions to elliptic equations with $L^1$ data. Moreover, the method of Kilpel\"{a}inen and Mal\'y does not require the strong summability condition on the right-hand side $f$, namely $f\in L^t(\Omega)$ with $t>\dfrac{N}{p}$, and it remains valid for $L^1$ data.

The question of continuity of solutions of anisotropic operators with measurable and bounded coefficients, behaving as in \eqref{eq1.2}, remains a significant challenge even after nearly seventy years. For a comprehensive discussion of the subject and relevant references, we refer to the surveys \cite{MajVes, MinRad}. The counterexamples constructed by Giaquinta \cite{Gia} and Marcellini \cite{Mar1} demonstrate that solutions may be unbounded if the values of $p_i$ differ too significantly. The local boundedness of solutions to anisotropic elliptic  equations has been proved in \cite{Kol, BocMarSbo, CupMarMas}. It is worth noting that most of the known regularity results for anisotropic equations pertain to operators with regular (Lipschitz continuous) coefficients. For instance, see the pioneering work \cite{Mar2} as well as \cite{BouBra, BouBraLeo} and the references therein. Regarding $C^0$ regularity, surprisingly, the most extensively studied case is the anisotropic porous medium equation, both with and without lower order terms \cite{BurSkr, Hen1, Hen2, Hen3} (see also \cite{CiaGua, CiaMosVes, PirRagVes}). To the best of our knowledge, for anisotropic elliptic equations, partial results can be found in \cite{LisSkr, LiaSkrVes, CiaSkrVes} under the precise conditions on the values $p_i$:
\begin{equation*}
\begin{cases}
p_1=2 <p=p_2= ...=p_N,\\
\text{or}\\
p=p_1=...=p_s<2=p_{s+1}=...=p_N,\quad 2p+s(p-2)>0.
\end{cases}
\end{equation*}
At the same time, in \cite{DiBGiaVes}, the authors employ a more general approach, though their results are obtained under a qualitative condition on the gap between the values $p_i$, specifically if $p_N-p_1 \ll 1$.

Equation \eqref{eq1.1} without lower order term was studied in \cite{CiaHenSkr}, using the strategy from \cite{Skr1, Skr2, SkrVoi}, under the condition $\sum\limits_{i=1}^N \frac{1}{p_i}=1$, with no additional restrictions on the values $p_i$. In this paper, we establish the continuity of solutions to \eqref{eq1.1}-\eqref{eq1.2} by following the similar strategy but employing a different iteration, specifically,  the Kilpel\"{a}inen--Mal\'y technique \cite{KilMal}, properly adapted to the anisotropic equations. 

%%%%%%%%%%%%%%%%%%%%%%%%%%%%%%%%%%%%%%%%%%%%%%%%%%%%%%%%%%%%%%%%%%%%%%

Before formulating the main results, let us remind the reader of the definition of a weak solution to equation \eqref{eq1.1}.
We extend the function $f$ by zero outside $\Omega$, i.e., on
  $\mathbb{R}^N\setminus \Omega$, and, for any $\rho >0$, set
\begin{equation}\label{eq1.5}
\mathcal{W}^{|f|}_{1,N}(\rho):=\sup\limits_{x\in \Omega}\mathcal{W}^{|f|}_{1,N}(x, \rho).
\end{equation}
Define the anisotropic Sobolev spaces $W^{1,\boldsymbol{p}}(\Omega)$ and $W^{1,\boldsymbol{p}}_0(\Omega)$ as follows
\begin{equation*}
\begin{aligned}
W^{1,\boldsymbol{p}}(\Omega):=\{u\in W^{1,1}(\Omega),\quad u_{x_i} \in L^{p_i}(\Omega),\quad i=1, ..., N\},\\
W^{1,\boldsymbol{p}}_0(\Omega):=\{u\in W^{1,1}_0(\Omega),\quad u_{x_i} \in L^{p_i}(\Omega),\quad i=1, ..., N\}.
\end{aligned}
\end{equation*}
\begin{definition} 
We say that a function $u\in W^{1,\boldsymbol{p}}(\Omega)\cap L^{\infty}(\Omega)$ is a bounded weak solution of equation \eqref{eq1.1} under conditions \eqref{eq1.2} if the following identity holds
\begin{equation}\label{eq1.6}
\sum\limits_{i=1}^N \int\limits_{\Omega}\,a_i(x, \nabla u) \frac{\partial \varphi}{\partial x_i} \,dx=\int\limits_{\Omega} f(x)\,\varphi dx,
\end{equation}
for arbitrary $\varphi\in W^{1,\boldsymbol{p}}_0(\Omega)\cap L^{\infty}(\Omega)$,\, $\varphi \geqslant 0$.
\end{definition}

Our main result in this paper reads as follows.

\begin{theorem}\label{th1.1}
Let $u$ be a bounded weak solution of \eqref{eq1.1}-\eqref{eq1.2}, and assume also that
\begin{equation}\label{eq1.7}
\lim\limits_{\rho\rightarrow 0} \mathcal{W}^{|f|}_{1, N}(\rho)=0,
\end{equation}
then $u \in C(\Omega)$.
\end{theorem}
\begin{remark}
In the case $p_1=...=p_N=N$ and $f=0$ or $f\neq 0$, we cover the known results, see, for example \cite{Mor, AlbCiaSbo, AlbFer, IwaOnn, JiaKosYan}.
\end{remark}

Now, a few words concerning the proof of the above theorem. First, for a fixed $x_0\in \Omega$ and $\rho>0$, we define the rescaled cubes
\begin{equation*}
B^{(\omega)}_\rho(x_0):=\big\{x: |x_i-x_{0,i}|\leqslant \frac{\rho^{\frac{p_N}{p_i}}}{\omega^{\frac{p_N-p_i}{p_i}}},\quad i=1,..., N\big\},\quad B^{(\omega)}_{8 \rho}(x_0)\subset \Omega,
\end{equation*}
where
\begin{equation*}
\mu^+:=\sup\limits_{B^{(\omega)}_\rho(x_0)}u,\quad \mu^-:=\inf\limits_{B^{(\omega)}_\rho(x_0)}u,\quad \mu^+-\mu^-\leqslant \omega < 2 M=2\sup\limits_{\Omega}|u|.
\end{equation*}
Such cubes are similar to the rescaled cylinders introduced by DiBenedetto for the parabolic case \cite{DiB}. Note that
\begin{equation*}
B^{(\omega)}_\rho(x_0)\subset B_\rho(x_0):=\big\{x:|x_i-x_{i,0}|\leqslant \rho,\quad i=1, ..., N\big\},
\end{equation*}
provided that $\omega\geqslant \rho$. By adapting the iteration technique of Kilpel\"{a}inen, Mal\'y  \cite{KilMal}, we establish a pointwise upper bound for the function
\begin{equation*}
v^{\pm}:=\frac{\omega}{u^{\pm}+\theta(\rho)},
\end{equation*}
with some   $\rho\leqslant \theta(\rho)\leqslant 1$, $\lim\limits_{\rho\rightarrow 0} \theta(\rho)=0$. The choice $u^+:=\mu^+-u$ or $u^-:=u-\mu^-$ depends on whether the solution $u$ attains large values compared to $\mu^+$ or $\mu^-$  in a relatively large set. Classically, the
idea of using  $log\,\, v^{\pm}$ to prove the H\"{o}lder continuity  of solutions to second-order elliptic
equations dates back to the seminal work of Moser \cite{Mos1}. For higher-order equations, a similar idea was later
adopted in \cite{Skr1, Skr2, SkrVoi} and subsequently applied to  elliptic anisotropic equations without lower order terms \cite{CiaHenSkr}. Here, we demonstrate how this strategy can be adapted to the fully anisotropic case with a singular lower order term, following the ideas  of Kilpel\"{a}inen--Mal\'y.

Having established the local boundedness of the function $v^{\pm}$, we  apply the weighted Poincar\'e inequality (see Lemma \ref{lem2.2} below) to obtain the uniform bound 
$$\sup\limits_{B^{(\omega)}_{\frac{\rho}{2}}(x_0)}\,\, v^{\pm}\leqslant C,$$
with some constant $C>0$ independent of $\rho$, $\omega$,
provided that  $\theta(\rho)$ is chosen according to the condition
$$[\theta(\rho)]^{\gamma_0}:=\sum\limits_{i=1}^N \Big(\int\limits_{B_{\rho}(x_0)}|u_{x_i}|^{p_i}\,dx\Big)^{\frac{1}{p_i}}+\mathcal{W}^{|f|}_{1,N}(2 \rho)+\rho,$$
with some sufficiently large $\gamma_0>0$. Consequently, we obtain
\begin{equation*}
\osc\limits_{B^{(\omega)}_{\frac{\rho}{2}}(x_0)}\,u\leqslant \big(1-\frac{1}{C}\big)\,\omega +\theta(\rho),
\end{equation*}
from which, by iteration, the required Theorem \ref{th1.1} follows.

The rest of the paper is devoted to the proof of Theorem \ref{th1.1}

\vskip 2mm
{\bf Acknowledgements.}  
The first author is supported by a grant of the National Academy of Sciences of Ukraine (project numbers is 0125U001647).

\section{Preliminaries}

In this section, we present two auxiliary lemmas that will be used throughout the paper. The first is the Sobolev embedding lemma \cite{Ven, Tro}.

\begin{lemma}\label{lem2.1}
Let $E$ be an open bounded set in $\mathbb{R}^N$ and let $u\in W^{1,\boldsymbol{p}}_0(E)$. Then, for any $q> 1$, the following inequality holds:
\begin{equation}\label{eq2.1}
\|u\|_{L^q(E)}\leqslant c |E|^{\frac{1}{q}}\,\prod\limits_{i=1}^N \|u_{x_i}\|^{\frac{1}{N}}_{L^{p_i}(E)},\quad \sum\limits_{i=1}^N \frac{1}{p_i}=1,
\end{equation}
where $c$ is a positive constant that depends only on $N$, $p_1$, ..., $p_N$.
\end{lemma}

The next lemma introduces well-known weighted Poincar\'e inequality (see \cite[Proposition~2.1]{DiB}).

\begin{lemma}\label{lem2.2}
Let $u \in W^{1,q}(B_{\rho}(y))$ for some $y\in \mathbb{R}^N$,  $\rho>0$, and $q \geqslant 1$. Then, there exists a constant $\gamma>0$, independent of $u$ and $\rho$, such that the following inequality holds:
\begin{equation}\label{eq2.2}
    \|u\|_{L^{q}(B_{\rho}(y))} \leqslant \gamma \quad \frac{|B_{\rho}(y)|}{\big|\{u=0\} \cap B_{\rho}(y)\big|^{\frac{N-1}{N}}} \| \nabla u\|_{L^{q}(B_{\rho}(y))}.
\end{equation}
\end{lemma}

\section{Oscillation Estimates}

For fixed $x_0\in \Omega$, $\rho$, $\omega >0$ construct the cube
$$
B^{(\omega)}_\rho(x_0):=\big\{x: |x_i-x_{0,i}|\leqslant \frac{\rho^{\frac{p_N}{p_i}}}{\omega^{\frac{p_N-p_i}{p_i}}},\quad i=1,..., N\big\},\quad B^{(\omega)}_{8 \rho}(x_0)\subset \Omega
$$
and assume also that
$$
\mu^+:=\sup\limits_{B^{(\omega)}_\rho(x_0)}u,\quad \mu^-:=\inf\limits_{B^{(\omega)}_\rho(x_0)}u,\quad \mu^+-\mu^-\leqslant \omega < 2 M.
$$

 To prove Theorem~\ref{th1.1}, which addresses the continuity of the solution to equation \eqref{eq1.1}--\eqref{eq1.2}, it is crucial to establish appropriate estimates for the oscillation of the solution. Proposition~\ref{pr3.1} plays a key role in this analysis.

First, we introduce the following value $\theta(\rho)\in[\rho, 1]$, which will be used frequently throughout this section:
\begin{equation}\label{eq3.1}
\big[\theta(\rho)\big]^{\gamma_0}:=\sum\limits_{i=1}^N \biggr(\int\limits_{B_{\rho}(x_0)}|u_{x_i}|^{p_i}\,dx\biggr)^{\frac{1}{p_i}}+ \mathcal{W}^{|f|}_{1,N}(2\rho)+\rho,
\end{equation}
where $B_{\rho}(x_0):=\{x:|x_i-x_{i,0}|<\rho, i=1, ..., N\}$ and $\gamma_0$ is a fixed positive constant to be determined later. 

\begin{propo}\label{pr3.1}
Let $u$ be a bounded weak solution to \eqref{eq1.1}--\eqref{eq1.2}. Then, there exist numbers $\beta\in (0,1)$ and $C_* >1$ depending only on the data such that either
\begin{equation}\label{eq3.2}
\omega \leqslant C_* \,\theta(\rho),
\end{equation}
or
\begin{equation}\label{eq3.3}
\osc\limits_{B^{(\omega)}_{\frac{\rho}{2}}(x_0)} u\leqslant \beta\,\omega.
\end{equation}
\end{propo}

 The rest of this section is dedicated to the proof of Proposition~\ref{pr3.1}, which relies on the well-known Kilpel\"{a}inen--Mal\'y technique introduced in \cite{KilMal}. To apply this method, we begin in Section~3.1 by performing a change of variables. In 3.2, we establish and prove the key analytical tool -- local energy estimates -- for the resulting rescaled functions. Section~3.3 identifies and describes two possible alternatives that arise during the proof. In 3.4, we construct a sequence of levels $k_j$, a central element of the Kilpel\"{a}inen--Mal\'y iteration technique. The core result of this technique, Lemma~\ref{lem3.2}, is presented and proved in 3.5. Finally, in Sections 3.6--3.8, we complete the proof of Proposition~\ref{pr3.1}, providing the necessary estimates.

\subsection{ Change of Variables}
We introduce the change of variables 
\begin{equation*}
x_i \rightarrow (x_i-x_{i, 0})\frac{\omega^{\frac{p_N-p_i}{p_i}}}{\rho^{\frac{p_N}{p_i}}},\quad i=1, ..., N
\end{equation*}
which maps the cube $B^{(\omega)}_{\rho}(x_0)$ into
$B_1:=\{x : |x_i| <1,\quad i=1, ..., N\}.$
Denoting again with $x$ the new variables, the rescaled functions

$$\tilde{u}^{-}:=\frac{\tilde{u}-\mu^-}{\omega},\quad \text{and}\quad  \tilde{u}^{+}:=\frac{\mu^+-\tilde{u}}{\omega}$$
are bounded nonnegative weak solutions to the equation
\begin{equation}\label{eq3.4}
-\sum\limits_{i=1}^N \left(\tilde{a}_i(x, \nabla \tilde{u}^{\pm})\right)_{x_i}=\tilde{f}(x),\quad x\in B_1,
\end{equation}
where $\tilde{f}(x):=\Big(\dfrac{\rho}{\omega}\Big)^{p_N} \big(x_{1,0}+x_1\frac{\rho^{\frac{pN}{p_1}}}{\omega^{\frac{p_N-p_1}{p_1}}},...,x_{N,0}+x_N\rho\big)$ and $\tilde{a}_i$ satisfy the conditions
\begin{equation}\label{eq3.5}
\begin{cases}
        \sum\limits_{i=1}^N \tilde{a}_i(x, \nabla \tilde{u}^{\pm}) \, \frac{\partial \tilde{u}^{\pm}}{\partial x_i} \geqslant K_1 \sum\limits_{i=1}^N \big|\frac{\partial \tilde{u}^{\pm}}{\partial x_i} \big|^{p_i},\\
        |\tilde{a}_i(x, \nabla \tilde{u}^{\pm})|\leqslant K_2 \Big(\sum\limits_{j=1}^N\,\big|\frac{\partial \tilde{u}^{\pm}}{\partial x_j}\big|^{p_j}\Big)^{1-\frac{1}{p_i}},\quad i=1,...,N.
    \end{cases}
\end{equation}
Moreover, the following integral identity holds
\begin{equation}\label{eq3.6}
\sum\limits_{i=1}^N \int\limits_{B_1} \tilde{a}_i(x, \nabla \tilde{u}^{\pm})\,\nabla \varphi\,dx=\int\limits_{B_1}\tilde{f}\,\varphi\,dx
\end{equation}
for any nonnegative function $\varphi \in W^{1,\boldsymbol{p}}_0(B_1)\cap L^{\infty}(B_1)$.

\subsection{Local Energy Estimates}
The following notation will be used in the sequel
$$ v^{\pm}(x):=\left(\frac{1}{\tilde{u}^{\pm}(x)+\theta(\rho)/\omega}-4\right)_{+}=\max\left\{\left(\frac{1}{\tilde{u}^{\pm}(x)+\theta(\rho)/\omega}-4\right), 0\right\}.$$
\begin{lemma}\label{lem3.1}
Let $\tilde{u}^{\pm}$ be bounded nonnegative weak solutions to \eqref{eq3.4}. For any $\lambda > 0$,  $m\geqslant p_N$,  any $k$, $d >0$, any $0<r \leqslant \dfrac{1}{2}$, any $0<\theta(\rho)\leqslant 1$, any set 
$$B^{(d)}_r(y):=\big\{x: |x_i-y_i|< \frac{r^{\frac{p_N}{p_i}}}{d^{\frac{p_N-p_i}{p_i}}},\quad i=1,..., N\big\}\subset B_1,$$
and any $\zeta \in C^1_0( B^{(d)}_r(y))$, $0\leqslant \zeta \leqslant 1$, $\zeta=1$ in $B^{(d)}_{\frac{r}{2}}(y)$, $|\zeta_{x_i}|\leqslant 2^{\frac{p_N}{p_i}+1} \dfrac{d^{\frac{p_N-p_i}{p_i}}}{r^{\frac{p_N}{p_i}}}$, $i=1, ..., N$ there exists  a constant $\gamma >0$ depending only on $\lambda$, $m$, $N$, $p_1$,..., $p_N$, $K_1$, $K_2$ and $M:=\sup\limits_{\Omega} |u|$, such that if
$$\omega \geqslant \theta(\rho),$$
then
\begin{multline}\label{eq3.7}
\sum\limits_{i=1}^N\,\int\limits_{B^{(d)}_{r}(y)\cap\{v^{\pm}>k\}}\frac{|v^{\pm}_{x_i}|^{p_i}}{(1+\frac{v^{\pm}-k}{d})^{1+\lambda}}\zeta^m\,dx
 \\\leqslant\gamma \Big(\frac{d}{r}\Big)^{p_N}\,[\theta(\rho)]^{-2(p_N-1)}\, \int\limits_{B^{(d)}_r(y)\cap\{v^{\pm}>k\}}\Big(1+\frac{v^{\pm}-k}{d}\Big)^{(1+\lambda)(p_N-1)}\,\zeta^{m-p_N}\,dx\\+\gamma d\,[\theta(\rho)]^{-2(p_N-1)}\int\limits_{B^{(d)}_r(y)} |\tilde{f}(x)|\,\zeta^m dx.
\end{multline}

\end{lemma}
\begin{proof}
Test identity \eqref{eq3.6} by
$$\varphi=\left[1-\left(1+\frac{(v^{\pm}-k)_+}{d}\right)^{-\lambda}\right]\,\zeta^m,$$
using conditions \eqref{eq3.5} and the fact that $0\leqslant \varphi \leqslant 1,$ we obtain
\begin{multline*}
\frac{\lambda K_1 }{d}\sum\limits_{i=1}^N\,\int\limits_{B^{(d)}_{r}(y)\cap\{v^{\pm}>k\}}\frac{|\tilde{u}^{\pm}_{x_i}|^{p_i}\,[\tilde{u}^{\pm}+\theta(\rho)/\omega]^{-2}}{\big(1+\frac{v^{\pm}-k}{d}\big)^{1+\lambda}}\zeta^m\,dx\\\leqslant K_2\,m \sum\limits_{i=1}^N\,\int\limits_{B^{(d)}_{r}(y)\cap\{v^{\pm}>k\}} \left(\sum\limits_{j=1}^N|\tilde{u}^{\pm}_{x_j}|^{p_j}\right)^{1-\frac{1}{p_i}}\,|\zeta_{x_i}|\zeta^{m-1}\,dx + \int\limits_{B^{(d)}_r(y)} |\tilde{f}|\,\zeta^m\, dx.
\end{multline*}
Applying Young's inequality, we arrive at the estimate
\begin{multline*}
\gamma\sum\limits_{i=1}^N\,\int\limits_{B^{(d)}_{r}(y)\cap\{v^{\pm}>k\}}\frac{|v^{\pm}_{x_i}|^{p_i}\,[\tilde{u}^{\pm}+\theta(\rho)/\omega]^{2(p_i-1)}}{\big(1+\frac{v^{\pm}-k}{d}\big)^{1+\lambda}}\zeta^m\,dx\\=
\sum\limits_{i=1}^N\,\int\limits_{B^{(d)}_{r}(y)\cap\{v^{\pm}>k\}}\frac{|u_{x_i}|^{p_i}\,[\tilde{u}^{\pm}+\theta(\rho)/\omega]^{-2}}{\big(1+\frac{v^{\pm}-k}{d}\big)^{1+\lambda}}\zeta^m\,dx\\\leqslant \gamma \Big(\frac{d}{r}\Big)^{p_N}\sum\limits_{i=1}^N\int\limits_{B^{(d)}_{r}(y)\cap\{v^{\pm}>k\}}[\tilde{u}^{\pm}+\theta(\rho)/\omega]^{2(p_i-1)}\Big(1+\frac{v^{\pm}-k}{d}\Big)^{(1+\lambda)(p_i-1)}\,\zeta^{m-p_N}\,dx\\+\gamma d\,\int\limits_{B^{(d)}_r(y)} |\tilde{f}|\,\zeta^m\, dx.
\end{multline*}
By the fact that $0\leqslant \tilde{u}^{\pm}\leqslant \dfrac{\mu^+-\mu^-}{\omega}\leqslant 1$,  $\dfrac{\theta(\rho)}{2M}\leqslant \dfrac{\theta(\rho)}{\omega}\leqslant 1$ and 
$$\frac{\theta(\rho)}{2M} \leqslant \tilde{u}^{\pm}(x)+\theta(\rho)/\omega \leqslant 2,\quad x\in B^{(d)}_r(y),$$
and since $p_1\leqslant p_2\leqslant ... \leqslant p_N$ it follows from the previous inequality that we obtain the required \eqref{eq3.7}, this proves the lemma.
\end{proof} 

\subsection{Alternatives}
We can always assume that $\mu^+-\dfrac{\omega}{4}\geqslant \mu^-+\dfrac{\omega}{4}$, since otherwise, inequality \eqref{eq3.3} holds with $\beta=\dfrac{1}{2}$. The  proof of Proposition \ref{pr3.1} is based on studying of two different cases, either
\begin{equation}\label{eq3.8}
|B_1\cap\{\tilde{u}^{-}\leqslant \frac{1}{4}\}|\leqslant \frac{1}{2}|B_1|,
\end{equation}
or
\begin{equation}\label{eq3.9}
|B_1\cap\{\tilde{u}^{+}\leqslant \frac{1}{4}\}|\leqslant |B_1\cap\{\tilde{u}^{-}\geqslant \frac{1}{4}\}|\leqslant \frac{1}{2}|B_1|.
\end{equation}
In the subsequent considerations, assume that inequality \eqref{eq3.2} is violated, i.e.
\begin{equation}\label{eq3.10}
\omega \geqslant C_*\,\theta(\rho).
\end{equation}
The proof under alternatives \eqref{eq3.8} or \eqref{eq3.9} follows a similar approach, so we consider only the first one. As mentioned in the section \ref{Introduction}, we use the Kilpel\"{a}inen-Mal\'y technique, which begins with defining the sequences of level lines $\{k_j\}_{j\in \mathbb{N}}$ and $\{\delta_j\}_{j\in \mathbb{N}}$.

\subsection{ Definition of the Sequences $\{k_j\}_{j\in \mathbb{N}}$ and $\{\delta_j\}_{j\in \mathbb{N}}$}\label{sec3.4}

 In this section, we construct the sequence of levels $k_j$, which plays a central role in the Kilpel\"{a}inen--Mal\'y iteration technique, along with the corresponding sequence of step sizes $\delta_j$ associated with $k_j$.

For $j \geqslant 0$ and  $k> k_j$, $k_j$ to be defined, set $r_j:=\dfrac{1}{2^{j+1}}$, $\delta_j(k):=k-k_j$, $B_{j, k}:=B^{(\delta_j(k))}_{r_j}(y)$ and let
\begin{equation}\label{eq3.11}
A_j(k):=\big[\theta(\rho)\big]^{-\gamma_1}\frac{[\delta_j(k)]^{p_N-N}}{r^{p_N}_j}\int\limits_{B_{j,k}}\Big(\frac{v^--k_j}{\delta_j(k)}\Big)^{(1+\lambda)(p_N-1)}_{+}\,\zeta^{m-p_N}_{j, k}\,dx,
\end{equation} 
where  $v^-$ was defined in Lemma \ref{lem3.1} and $\gamma_1$, $\lambda$, $m$ are some fixed positive numbers to be specified later
and $\zeta_{j, k} \in C^1_0\big(B_{j,k}\big)$, $0\leqslant \zeta_{j, k} \leqslant 1$, $\zeta_{j, k}=1$ in $B^{(\delta_j(k))}_{r_{j+1}}(y)$ and $\Big|\dfrac{\partial \zeta_{j,k}}{\partial x_i}\Big|^{p_i}\leqslant \gamma \dfrac{[\delta_j(k)]^{p_N-p_i}}{r^{p_N}_j}$.

\subsubsection{ Definition of $k_1$ and $\delta_0$}\label{3.4.1}
Fix a  number $\varkappa \in (0,1)$ to be specified later depending only on the data. Let $y\in B_{\frac{1}{2}}$ be arbitrary. We set $k_{-1}:=0$ and $k_0:=\max\big(\mathcal{I}, 1)$, where
$$
\mathcal{I}:=\bigg(\frac{\big[\theta(\rho)\big]^{-\gamma_1}}{\varkappa}\int\limits_{B_{\frac{1}{2}}(y)}\big[v^-\big]^{(1+\lambda)(p_N-1)}\,dx\bigg)^{\frac{1}{N-1+\lambda(p_N-1)}}.
$$
Correspondingly, we introduce $\delta_{-1}:=\delta_{-1}(k_0)=k_0-k_{-1}=\max\big(\mathcal{I}, 1\big)$.
If 
$$
A_0(k_0+\tfrac{1}{2}\,\delta_{-1})\leqslant \varkappa,
$$
then we set $k_1:=k_0+\frac{1}{2}\,\delta_{-1}$. Note that $A_0(k)$ is a continuous, decreasing function, and $A_0(k) \downarrow 0$ as $k\rightarrow \infty$. Thus, if
$$
A_0(k_0+\tfrac{1}{2}\,\delta_{-1}) >\varkappa, 
$$
there exists $\bar{k}> k_0+\frac{1}{2}\,\delta_{-1}$ such that $A_0(\bar{k})=\varkappa$. In this case, we set $k_1:=\bar{k}$, and in both scenarios, we define $\delta_0:=\delta_{0}(k_1)=k_1-k_0$.

\begin{claim}
    There exists $c_0 >1$ depending only on the data and $\lambda$ such that
\begin{equation}\label{eq3.12}
\delta_0\leqslant c_0\,\,\delta_{-1}.
\end{equation}
\end{claim}
\begin{proof}
Indeed, if $k_1=k_0+\frac{1}{2}\,\delta_{-1}$, then inequality \eqref{eq3.12} is evident. Now, if $k_1>k_0+\frac{1}{2}\,\delta_{-1}$, 
inequality $\delta_{-1} \geqslant 1$ implies that $B^{(c_0\delta_{-1})}_{r_0}(y)\subset B_{\frac{1}{2}}(y)$. By our choice of $\mathcal{I}$, we have
\begin{equation*}
\begin{aligned}
A_0(k_0+c_0\,\delta_{-1})&\leqslant 2^{p_N}\big[\theta(\rho)\big]^{-\gamma_1}(c_0\,\mathcal{I})^{1-N-\lambda(p_N-1)} \int\limits_{B_{\frac{1}{2}}(y)} \big[v^-\big]^{(1+\lambda)(p_N-1)}\,dx
\\&\leqslant \frac{2^{p_N}\,\varkappa}{c^{N-1+\lambda(p_N-1)}_0}=\frac{\varkappa}{2},
\end{aligned}
\end{equation*}
provided that $c_0:=2^{\frac{p_N+1}{N-1+\lambda(p_N-1)}}$. Since $A_0(k_1)=\varkappa$ and $A_0(k)$ is a decreasing function, it follows that $k_1 <k_0+c_0\,\delta_{-1}$, from which the required \eqref{eq3.12} follows.
\end{proof}

\subsubsection{Induction Argument}\label{sec3.4.2}

We construct the sequences of positive numbers $\{k_j\}_{j\in \mathbb{N}}$ and $\{\delta_j\}_{j\in \mathbb{N}}$ inductively as follows. Suppose that $k_1$, \dots, $k_j$ and $\delta_0$, \dots, $\delta_{j-1}$ have already been chosen such that
\begin{equation}\label{eq3.13}
A_i(k_{i+1})\leqslant \varkappa,\quad \text{and}\quad \delta_i\leqslant c_0\,\delta_{i-1},\quad i=0, ..., j-1.
\end{equation}
We now proceed to determine how to choose $k_{j+1}$ and $\delta_j$.
If 
$$
A_j(k_j+\frac{1}{2}\,\delta_{j-1})\leqslant \varkappa, 
$$
then we set $k_{j+1}= k_j+\frac{1}{2}\,\delta_{j-1}$. Note that $A_j(k) \downarrow 0$ as $k\rightarrow \infty$. Therefore, if 
$$
A_j(k_j+\frac{1}{2}\,\delta_{j-1}) > \varkappa, 
$$
there exists $\bar{k}> k_j+\frac{1}{2}\,\delta_{j-1}$ such that $A_j(\bar{k}) =\varkappa$. In this case, we set $k_{j+1}:=\bar{k}$. In both cases, we define $\delta_j:=\delta_j(k_{j+1})=k_{j+1}-k_j$. Note that our choices ensure that
\begin{equation}\label{eq3.14}
A_j(k_{j+1})\leqslant \varkappa.
\end{equation}
Moreover, since $B_{j, k_j+c_0 \delta_{j-1}}\subset B_{j-1, k_{j}}$ and $\big\{\zeta_{j, k_j+c_0\delta_{j-1}}\neq 0\big\}\subset \big\{\zeta_{j-1, k_j}=1\big\}$, by \eqref{eq3.13} and our choice of $c_0$, we have
\begin{equation*}
A_j(k_j+ c_0\,\delta_{j-1})\leqslant \frac{2^{p_N}}{c^{N-1+\lambda(p_N-1)}_0}\,A_{j-1}(k_j)\leqslant \frac{2^{p_N}\,\varkappa}{c^{N-1+\lambda(p_N-1)}_0}\,=\frac{\varkappa}{2},
\end{equation*}
which implies
\begin{equation}\label{eq3.15}
\delta_j\leqslant c_0\, \delta_{j-1}.
\end{equation}
In what follows, we write $B_j:=B_{j, k_{j+1}}=B^{(\delta_j(k_{j+1}))}_{r_j}(y)$ and $\zeta_j:=\zeta_{j, k_{j+1}}$.

\subsection{Kilpel\"{a}inen--Mal\'y Lemma}
The following lemma is the key to the Kilpel\"{a}inen--Mal\'y technique. 
\begin{lemma}\label{lem3.2}
Let the sequences $r_j$, $k_j$, and $\delta_j$ be defined in section~\ref{sec3.4}.
Additionally, let $0 < \lambda < p_1 - 1$. Then for any $j\geqslant 1$, there holds
\begin{equation}\label{eq3.16}
\delta_j\,\leqslant \frac{1}{2}\,\delta_{j-1} +\gamma r_j+ \gamma \Big(\big[\theta(\rho)\big]^{-\gamma_1}\,\int\limits_{B_{r_{j-1}}(y)} |\tilde{f}(x)|\,dx \Big)^{\frac{1}{N-1}},
\end{equation}
where $B_{r_{j-1}}(y)=\{x:|x_i-y_{i}|<r_{j-1}, \quad i=1, ...,  N\}$.
\end{lemma}

\begin{proof}
Fix $j\geqslant 1$ and assume that
\begin{equation}\label{eq3.17}
\delta_j > \frac{1}{2}\delta_{j-1}+ r_j,
\end{equation}
since otherwise, the estimate \eqref{eq3.16} is evident. Inequality \eqref{eq3.17}  ensures that $k_{j+1} > k_j + \tfrac{1}{2} \delta_{j-1}$, which, according to the induction arguments described in subsection~\ref{sec3.4.2}, implies that $A_j(k_{j+1}) = \varkappa$. Moreover,
$$\frac{r^{\frac{p_N}{p_i}}_j}{\delta^{\frac{p_N-p_i}{p_i}}_j}\leqslant \frac{1}{2} \frac{r^{\frac{p_N}{p_i}}_{j-1}}{\delta^{\frac{p_N-p_i}{p_i}}_{j-1}},\quad \text{and}\quad\frac{r^{\frac{p_N}{p_i}}_j}{\delta^{\frac{p_N-p_i}{p_i}}_j}\leqslant r_j,\quad i=1,..., N, $$
and hence 
\begin{equation}\label{eq3.18}
B_{j}\subset B_{j-1},\,\,\, B_{j}\subset B_{r_j}(y),\,\,\, \{x\in B_j:\,\, \zeta_j \neq 0\}\subset B_{j-1},\,\,\, \text{and}\,\,\, \{x\in B_j:\,\,\, \zeta_j \neq 0\} \subset B_{r_{j-1}}(y).
\end{equation}
Let us estimate the term on the right-hand side of \eqref{eq3.11} for $k=k_{j+1}$. To do this, we apply the decomposition $B_j\cap\{v^-> k_j\}= L_{\epsilon, j} \cup E_{\epsilon, j}$, where 
\begin{equation*}
L_{\epsilon, j}:=\big\{x\in B_j\cap\{v^-> k_j\} : \dfrac{v^--k_j}{\delta_j}\leqslant \epsilon\big\},\quad E_{\epsilon, j}:=B_j\cap\{v^-> k_j\} \setminus L_{\epsilon, j}.
\end{equation*}
Here, the parameter $\epsilon$ depends on the given data and is chosen to be sufficiently small, its precise value will be specified later.
 Thus, we write the decomposition for $A_j(k_{j+1})$ in the following form:
\begin{equation}\label{eq3.18*}
\begin{aligned}
\varkappa=A_j(k_{j+1})=&\big[\theta(\rho)\big]^{-\gamma_1}\frac{[\delta_j]^{p_N-N}}{r^{p_N}_j}\int\limits_{B_j\cap\{v^-> k_j\}}\Big(\frac{v^--k_j}{\delta_j}\Big)^{(1+\lambda)(p_N-1)}\,\zeta^{m-p_N}_{j, k}\,dx
\\=&\big[\theta(\rho)\big]^{-\gamma_1}\frac{[\delta_j]^{p_N-N}}{r^{p_N}_j}\int\limits_{L_{\epsilon, j}}\Big(\frac{v^--k_j}{\delta_j}\Big)^{(1+\lambda)(p_N-1)}\,\zeta^{m-p_N}_{j, k}\,dx
\\&+\big[\theta(\rho)\big]^{-\gamma_1}\frac{[\delta_j]^{p_N-N}}{r^{p_N}_j}\int\limits_{E_{\epsilon, j}}\Big(\frac{v^--k_j}{\delta_j}\Big)^{(1+\lambda)(p_N-1)}\,\zeta^{m-p_N}_{j, k}\,dx
\\=:& \,A_j^{(L_{\epsilon, j})}+A_j^{(E_{\epsilon, j})}
\end{aligned}
\end{equation}

 Now, we choose the constant $m$ from \eqref{eq3.11} such that $m\geqslant p_N$. Furthermore, note that for any $x\in B_j\cap\{v^->k_j\}$, the following inequality holds:
$$
1=\frac{k_j-k_{j-1}}{k_j-k_{j-1}}\leqslant \frac{v^--k_{j-1}}{\delta_{j-1}}.%,\quad x\in B_j\cap\{v^->k_j\}, 
$$
By applying inequalities \eqref{eq3.14} and \eqref{eq3.15}, we can now evaluate $A_j^{(L_{\epsilon, j})}$
\begin{equation}\label{eq3.19}
\begin{aligned}
A_j^{(L_{\epsilon, j})}&\leqslant\gamma\, \epsilon^{(1+\lambda)(p_N-1)}\,\big[\theta(\rho)\big]^{-\gamma_1} \frac{[\delta_{j-1}]^{p_N-N}}{r^{p_N}_{j-1}}\int\limits_{L_{\epsilon, j}}\,\zeta^{m-p_N}_{j-1}\,dx 
\\&\leqslant \gamma \epsilon^{(1+\lambda)(p_N-1)} \big[\theta(\rho)\big]^{-\gamma_1}\frac{[\delta_{j-1}]^{p_N-N}}{r^{p_N}_{j-1}}\int\limits_{B_{j-1}}\Big(\frac{v^--k_{j-1}}{\delta_{j-1}}\Big)^{(1+\lambda)(p_N-1)}_{+}\,\zeta^{m-p_N}_{j-1}\,dx
\\&=\gamma \epsilon^{(1+\lambda)(p_N-1)} A_{j-1}(k_j)\leqslant \gamma \epsilon^{(1+\lambda)(p_N-1)}\,\varkappa,
\end{aligned}
\end{equation}
 where the generic constant $\gamma$ depends on $c_0$ and $p_N$.
To derive an estimate for  $A_j^{(E_{\epsilon, j})}$, we set the constant $\lambda$ from \eqref{eq3.11} such that $\lambda <p_1-1$ and introduce the following auxiliary function
$$
W (\sigma):=\int\limits^{\sigma}_0\,(1+s)^{-\frac{1+\lambda}{p_1}}\,ds,\quad \sigma >0.
$$
Evidently, we have
\begin{equation}\label{eq3.20}
W(\sigma)\leqslant \gamma(\lambda, p_1)\,(1+\sigma)^{\frac{p_1-1-\lambda}{p_1}},\qquad\text{and}\qquad
 W(\sigma)\geqslant \gamma(\epsilon, \lambda, p_1)\,\sigma^{\frac{p_1-1-\lambda}{p_1}},\text{ if }\sigma \geqslant \epsilon.
\end{equation}

Now, we apply Lemma~\ref{lem2.1} with $q=\dfrac{p_1(p_N-1)(1+\lambda)}{p_1-1-\lambda}$ and choose the constant $\lambda$ such that $q>N$, which implies that $\frac{Np_1}{p_1(p_N-1)+N}-1<\lambda<p_1-1$. Furthermore, by applying Young's inequality, we obtain
\begin{equation}\label{eq3.21}
\begin{aligned}
A_j^{(E_{\epsilon, j})}\leqslant& \gamma(\epsilon)\big[\theta(\rho)\big]^{-\gamma_1}\frac{[\delta_j]^{p_N-N}}{r^{p_N}_j}\int\limits_{E_{\epsilon, j}}\Big[W\Big(\frac{v^--k_j}{\delta_j}\Big)\,\zeta^{\frac{m-p_N}{q}}_j\Big]^{q}\,dx
\\\leqslant&\gamma(\epsilon)\big[\theta(\rho)\big]^{-\gamma_1}\frac{[\delta_j]^{p_N-N}}{r^{p_N}_j}\int\limits_{B_{j}\cap\{v^->k_j\}} \Big[W\Big(\frac{v^--k_j}{\delta_j}\Big)\,\zeta^{\frac{m-p_N}{q}}_j\Big]^{q}\,dx
\\\leqslant&\gamma(\epsilon)\big[\theta(\rho)\big]^{-\gamma_1} \prod\limits_{i=1}^N \Bigg(\int\limits_{B_{j}\cap\{v^->k_j\}}\biggr|\frac{\partial}{\partial x_i}\Big[W\Big(\frac{v^--k_j}{\delta_j}\Big)\,\zeta^{\frac{m-p_N}{q}}_j\Big]\biggr|^{p_i} dx\Bigg)^{\frac{q}{N p_i}}
\\=&\frac{\gamma(\epsilon)}{\delta^{q}_j}\big[\theta(\rho)\big]^{-\gamma_1} \prod\limits_{i=1}^N \Bigg(\delta^{p_i}_j\,\int\limits_{B_{j}\cap\{v^->k_j\}}\biggr|\frac{\partial}{\partial x_i}\Big[W\Big(\frac{v^--k_j}{\delta_j}\Big)\,\zeta^{\frac{m-p_N}{q}}_j\Big]\biggr|^{p_i} dx\Bigg)^{\frac{q}{N p_i}}
\\\leqslant&\frac{\gamma(\epsilon)}{\delta^{q}_j}\big[\theta(\rho)\big]^{-\gamma_1}  \Bigg(\,\sum\limits_{i=1}^N\,\delta^{p_i}_j\,\int\limits_{B_{j}\cap\{v^->k_j\}}\biggr|\frac{\partial}{\partial x_i}W\Big(\frac{v^--k_j}{\delta_j}\Big)\biggr|^{p_i} \zeta^{\frac{(m-p_N)p_1}{q}}_j dx
\\&+\sum\limits_{i=1}^N\,\delta^{p_i}_j\,\int\limits_{B_{j}\cap\{v^->k_j\}} \biggr[W\Big(\frac{v^--k_j}{\delta_j}\Big)\biggr]^{p_i}\biggr|\frac{\partial \zeta_j}{\partial x_i}\biggr|^{p_i}\,\zeta^{\frac{(m-p_N)p_i}{q}-p_i}_j\,dx\Bigg)^{\frac{q}{N}}.
\end{aligned}
\end{equation}
Let us estimate the integrals on the right-hand side of \eqref{eq3.21}. We apply Lemma \ref{lem3.1} to the first integral, which then leads to the following bound
\begin{equation*}
\begin{aligned}
\delta^{-N}_j&\big[\theta(\rho)\big]^{-\frac{N}{q}\gamma_1}  \sum\limits_{i=1}^N\,\delta^{p_i}_j\,\int\limits_{B_{j}\cap\{v^->k_j\}}\Big|\frac{\partial}{\partial x_i}W\Big(\frac{v^--k_j}{\delta_j}\Big)\Big|^{p_i} \zeta^{\frac{(m-p_N)p_1}{q}}_j dx
\\\leqslant&\,\delta^{-N}_j\big[\theta(\rho)\big]^{-\frac{N}{q}\gamma_1}  \,\sum\limits_{i=1}^N\,\,\int\limits_{B_{j}\cap\{v^->k_j\}}\frac{|v^-_{x_i}|^{p_i}}{\big(1+\frac{v^--k_j}{\delta_j}\big)^{1+\lambda}}\zeta^{\frac{(m-p_N)p_1}{q}}_j dx
\\\leqslant&\,\gamma\,\frac{\delta^{p_N-N}_j}{r^{p_N}_j}\big[\theta(\rho)\big]^{-\frac{N}{q}\gamma_1 -2(p_N-1)}\int\limits_{B_{j}\cap\{v^->k_j\}}\Big(1+\frac{v^--k_j}{\delta_j}\Big)^{(1+\lambda)(p_N-1)}\zeta^{\frac{(m-p_N)p_1}{q}-p_N}_j dx
\\&+ \gamma\,\delta^{1-N}_j\big[\theta(\rho)\big]^{-\frac{N}{q}\gamma_1 -2(p_N-1)}\int\limits_{B_{j}}|\tilde{f}|\,\zeta^{\frac{(m-p_N)p_1}{q}}_j\,dx.
\end{aligned}
\end{equation*}
Choose $m$ from the condition $\dfrac{(m-p_N)p_1}{q}-p_N=1$ and fix $\gamma_1>0$ by the condition 
$$\dfrac{N}{q}\gamma_1+2(p_N-1)=\gamma_1,\quad \text{i.e.}\quad \gamma_1=\dfrac{2(p_N-1)q}{q-N}.$$ 
Using \eqref{eq3.14}, \eqref{eq3.15}, \eqref{eq3.17} and \eqref{eq3.18}, we obtain
\begin{equation}\label{eq3.220}
\begin{aligned}
&\big[\theta(\rho)\big]^{-\gamma_1}\,\frac{\delta^{p_N-N}_j}{r^{p_N}_j}\int\limits_{B_{j}\cap\{v^->k_j\}}\Big(1+\frac{v^--k_j}{\delta_j}\Big)^{(1+\lambda)(p_N-1)}\zeta_j\,dx
\\&\leqslant\,\gamma \big[\theta(\rho)\big]^{-\gamma_1}\,\frac{\delta^{p_N-N}_{j-1}}{r^{p_N}_{j-1}}\int\limits_{B_{j}\cap\{v^->k_j\}}\Big(\frac{v^--k_{j-1}}{\delta_{j-1}}\Big)^{(1+\lambda)(p_N-1)}\zeta_{j-1}^{m-p_N}\,dx
\\&\leqslant \gamma A_{j-1}(k_{j})\leqslant\gamma\,\varkappa.
\end{aligned}
\end{equation}

Therefore, we derive the following bound for the first integral on the right-hand side of  \eqref{eq3.21}
\begin{equation}\label{eq3.22}
\begin{aligned}
    &\delta^{-N}_j\big[\theta(\rho)\big]^{-\frac{N}{q}\gamma_1}  \sum\limits_{i=1}^N\,\delta^{p_i}_j\,\int\limits_{B_{j}\cap\{v^->k_j\}}\Big|\frac{\partial}{\partial x_i}W\Big(\frac{v^--k_j}{\delta_j}\Big)\Big|^{p_i} \zeta^{\frac{(m-p_N)p_1}{q}}_j dx
    \\&\leqslant \gamma\,\varkappa+\gamma\,\delta^{1-N}_j\big[\theta(\rho)\big]^{-\gamma_1}\int\limits_{B_{j}}|\tilde{f}|\,dx.
\end{aligned}
\end{equation}
The second term on the right-hand side of \eqref{eq3.21} is estimated using   \eqref{eq3.220}: 
\begin{equation}\label{eq3.23}
\begin{aligned}
 &\delta^{-N}_j\big[\theta(\rho)\big]^{-\frac{N}{q}\gamma_1}\sum\limits_{i=1}^N\,\delta^{p_i}_j\,\int\limits_{B_{j}\cap\{v^->k_j\}} \Big[W\Big(\frac{v^--k_j}{\delta_j}\Big)\Big]^{p_i}\Big|\frac{\partial \zeta_j}{\partial x_i}\Big|^{p_i}\,\zeta^{\frac{(m-p_N)p_i}{q}-p_i}_j\,dx
 \\&\leqslant \gamma\big[\theta(\rho)\big]^{-\gamma_1}\,\frac{\delta^{p_N-N}_j}{r^{p_N}_j}\int\limits_{B_{j}\cap\{v^->k_j\}}\Big(1+\frac{v^--k_j}{\delta_j}\Big)^{(1+\lambda)(p_N-1)}\zeta_j\,dx\leqslant \gamma\,\varkappa.
\end{aligned}
\end{equation}

Collecting estimates \eqref{eq3.18*}, \eqref{eq3.19}, \eqref{eq3.21}, \eqref{eq3.22}, \eqref{eq3.23}, we obtain
\begin{equation*}
\varkappa\leqslant \epsilon^{(1+\lambda)(p_N-1)}\varkappa +\gamma(\epsilon)\Bigg(\varkappa +\delta^{1-N}_j\,\big[\theta(x_0,\rho)\big]^{-\gamma_1}\,\int\limits_{B_{r_{j-1}}(y)}|\tilde{f}|\,dx\Bigg)^{\frac{q}{N}}.
\end{equation*}
Choosing $\epsilon$ and $\varkappa$ such that
$\epsilon^{(1+\lambda)(p_N-1)}=\frac{1}{4}$, and $\gamma(\epsilon)\varkappa^{\frac{q}{N}-1}=\frac{1}{4}$,
the last inequality yields
$$\delta_j\leqslant \gamma(\epsilon, \varkappa)\Big(\big[\theta(\rho)\big]^{-\gamma_1}\int\limits_{B_{r_{j-1}}(y)} |\tilde{f}|\,dx\Big)^{\frac{1}{N-1}},$$
which proves \eqref{eq3.16}. This completes the proof of the lemma.
\end{proof}

\subsection{Upper Bound of $v^-$}

Fix $J>1$ and sum inequality \eqref{eq3.16} over $j$ from $1$ to $J-1$. Using \eqref{eq3.12} and definition of $r_j$, we derive:
\begin{equation*}\label{eq3.24}
\begin{aligned}
k_J&\leqslant \gamma(c_0) \,\delta_0+ \gamma \sum\limits_{j=1}^{J-1}\,r_j+\gamma\sum\limits_{j=1}^{J-1}\Big(\big[\theta(\rho)\big]^{-\gamma_1}\int\limits_{B_{r_{j-1}}(y)}|\tilde{f}|\,dx\Big)^{\frac{1}{N-1}}
\\&\leqslant \gamma +\gamma \mathcal{I} +\gamma\sum\limits_{j=1}^{J-1}\Big(\big[\theta(\rho)\big]^{-\gamma_1}\int\limits_{B_{r_{j-1}}(y)}|\tilde{f}|\,dx\Big)^{\frac{1}{N-1}}.
\end{aligned}
\end{equation*}
Returning to the original variables in the last term and estimating the integral sum by the corresponding integral, using \eqref{eq3.10}, we obtain:
\begin{equation*}
\sum\limits_{j=1}^{J-1}\Big(\int\limits_{B_{r_{j-1}}(y)}|\tilde{f}|\,dx\Big)^{\frac{1}{N-1}}=\sum\limits_{j=1}^{J-1}\Big(\omega^{-N}\int\limits_{B_{\frac{\rho}{2^{j-1}}}(y)}|f|\,dx\Big)^{\frac{1}{N-1}}\leqslant \gamma \big[\theta(\rho)\big]^{-\frac{N}{N-1}} \mathcal{W}^{|f|}_{1, N}(2\rho).
\end{equation*}
Combining the last two inequalities, we get
\begin{equation}\label{eq3.24}
k_J\leqslant \gamma +\gamma \bigg(\big[\theta(\rho)\big]^{-\gamma_1}\int\limits_{B_{1}}\big[v^-\big]^{(1+\lambda)(p_N-1)}\,dx\bigg)^{\frac{1}{N-1+\lambda(p_N-1)}}+\gamma \big[\theta(\rho)\big]^{-\frac{\gamma_1+N}{N-1}} \mathcal{W}^{|f|}_{1,N}(2 \rho).
\end{equation}
Since the sequence $\{k_j\}_{j\in \mathbb{N}}$ is increasing by construction and bounded from above, it converges. Let $k_{\infty}:=\lim\limits_{j\rightarrow \infty} k_j$ denote its limit.
Moreover, as $j \rightarrow \infty$, we have  $\delta_j \rightarrow 0$,  allowing us to pass to the limit $J\rightarrow \infty$ in \eqref{eq3.24}. From \eqref{eq3.14} we conclude that
\begin{equation*}
 \frac{1}{r^{p_N}_j}\int\limits_{B_j}(v^{-}-k_{j})^{(1+\lambda)(p_N-1)}_{+}\,dx\leqslant \gamma \delta^{N-1+\lambda(p_N-1)}_j\rightarrow 0,\quad j\rightarrow \infty.
 \end{equation*}
Choosing $y$ as a Lebesque point of the function $v^-$ we conclude that $v^-(y) \leqslant k_\infty$ and hence $v^-(y)$ is estimated by the right-hand side of \eqref{eq3.24}. Applicability of the Lebesgue
differentiation theorem follows from \cite[Chap. II, Sec. 3]{Guz}.
  Since $y$ is an arbitrary point in $B_{\frac{1}{2}}$ (see \ref{3.4.1}), we obtain
\begin{equation}\label{eq3.25}
\begin{aligned}
\sup\limits_{B_{\frac{1}{2}}} v^-\leqslant& \,\gamma +\gamma \bigg(\big[\theta(\rho)\big]^{-\gamma_1}\int\limits_{B_{1}}\big[v^-\big]^{(1+\lambda)(p_N-1)}\,dx\bigg)^{\frac{1}{N-1+\lambda(p_N-1)}}
\\&+\gamma \big[\theta(\rho)\big]^{-\frac{\gamma_1+N}{N-1}} \mathcal{W}^{|f|}_{1,N}(2 \rho).
\end{aligned}
\end{equation}

\subsection{Uniform Bound for the Integral of $v^-$}

Now, we estimate the integral on the right-hand side of \eqref{eq3.25}. According to the first alternative \eqref{eq3.8}, we have
$$|B_1\cap\{v^-=0\}|=\biggr|B_1\cap\Big\{\frac{1}{\tilde{u}^{-}+\theta(\rho)/\omega}\leqslant 4\Big\}\biggr|\geqslant \biggr|B_1\cap\Big\{\tilde{u}^{-}\geqslant \frac{1}{4}\Big\}\biggr|\geqslant \frac{|B_1|}{2}.$$
Since $p_N>N\geqslant2$, we have $(1+\lambda)(p_N-1)>1$. Taking into account the definition of the function $v^-$ and the fact that  $\omega<2M$, we apply Lemma \ref{lem2.2} and obtain
\begin{equation*}
\int\limits_{B_1} \big[v^-\big]^{(1+\lambda)(p_N-1)}\,dx\leqslant
\gamma \big[\theta(\rho)\big]^{-(1+\lambda)(p_N-1)+1}\,\int\limits_{B_1} v^-\,dx\leqslant \gamma
\big[\theta(\rho)\big]^{-(1+\lambda)(p_N-1)+1}\,\sum\limits_{i=1}^N\,\int\limits_{B_1} |v^-_{x_i}|\,dx .
\end{equation*}
Returning to the original coordinates, using  the H\"{o}lder inequality and \eqref{eq3.10}, we obtain
\begin{equation*}
\begin{aligned}
\sum\limits_{i=1}^N\,\int\limits_{B_1} |v^-_{x_i}|\,dx&\leqslant \gamma\,\big[\theta(\rho)\big]^{-2}\sum\limits_{i=1}^N \rho^{-p_N(1-\frac{1}{p_i})}\,\omega^{1-\frac{p_N}{p_i}+p_N-N}\int\limits_{B^{(\omega)}_\rho(x_0)}|u_{x_i}|\,dx
\\&\leqslant \gamma\,\big[\theta(\rho)\big]^{-2}\sum\limits_{i=1}^N \,\omega^{1-\frac{N}{p_i}}\Big(\int\limits_{B^{(\omega)}_\rho(x_0)}|u_{x_i}|^{p_i}\,dx\Big)^{\frac{1}{p_i}}
\\&\leqslant \gamma \big[\theta(\rho)\big]^{-1-\frac{N}{p_1}}\,\sum\limits_{i=1}^N\,\Big(\int\limits_{B^{(\omega)}_\rho(x_0)}|u_{x_i}|^{p_i}\,dx\Big)^{\frac{1}{p_i}}.
\end{aligned}
\end{equation*}
Using the inequality $\theta(\rho) \geqslant \rho$ and \eqref{eq3.10}, we obtain $\omega > \rho$, which implies the embedding $B^{(\omega)}_\rho(x_0)\subset B_\rho(x_0)$.
Combining the last two inequalities, we obtain
\begin{equation}\label{eq3.26}
\begin{aligned}
\int\limits_{B_1} \big[v^-\big]^{(1+\lambda)(p_N-1)}\,dx&\leqslant
\gamma \big[\theta(\rho)\big]^{-(1+\lambda)(p_N-1)-\frac{N}{p_1}}\,\sum\limits_{i=1}^N\,\biggr(\int\limits_{B^{(\omega)}_\rho(x_0)}|u_{x_i}|^{p_i}\,dx\biggr)^{\frac{1}{p_i}}
\\&\leqslant \gamma \big[\theta(\rho)\big]^{-(1+\lambda)(p_N-1)-\frac{N}{p_1}}\,\sum\limits_{i=1}^N\,\biggr(\int\limits_{B_\rho(x_0)}|u_{x_i}|^{p_i}\,dx\biggr)^{\frac{1}{p_i}}.
\end{aligned}
\end{equation}

\subsection{Proof of Proposition \ref{pr3.1}}

Inequalities \eqref{eq3.25} and \eqref{eq3.26} yield
\begin{equation}\label{eq3.27}
\begin{aligned}
\sup\limits_{B_{\frac{1}{2}}} v^-\leqslant \gamma &+\gamma \Bigg(\big[\theta(\rho)\big]^{-\gamma_1-(1+\lambda)(p_N-1)-\frac{N}{p_1}}\sum\limits_{i=1}^N\,\Big(\int\limits_{B_\rho(x_0)}|u_{x_i}|^{p_i}\,dx\Big)^{\frac{1}{p_i}}\Bigg)^{\frac{1}{N-1+\lambda(p_N-1)}}
\\&+\gamma \big[\theta(\rho)\big]^{-\frac{\gamma_1+N}{N-1}} \mathcal{W}^{|f|}_{1,N}(2 \rho).
\end{aligned}
\end{equation}

Fixing $\gamma_0$ such that 
$$\gamma_0=\max\left\{\gamma_1+(1+\lambda)(p_N-1)+\frac{N}{p_1}, \frac{\gamma_1+N}{N-1}\right\}$$
 and using our choice \eqref{eq3.1} for $\theta(\rho)$, we obtain
\begin{equation}\label{eq3.28}
\sup\limits_{B_{\frac{1}{2}}} v^-\leqslant\gamma +\gamma \Bigg(\big[\theta(\rho)\big]^{-\gamma_0}\sum\limits_{i=1}^N\,\Big(\int\limits_{B_\rho(x_0)}|u_{x_i}|^{p_i}\,dx\Big)^{\frac{1}{p_i}}\Bigg)^{\frac{1}{N-1+\lambda(p_N-1)}}
+\gamma \big[\theta(\rho)\big]^{-\gamma_0} \mathcal{W}^{|f|}_{1,N}(2 \rho)
\leqslant C,
\end{equation}
with a positive constant $C$ depending only on the data. Inequalities  \eqref{eq3.10} and  \eqref{eq3.28} imply
\begin{equation*}
\inf\limits_{B^{(\omega)}_{\frac{\rho}{2}}(x_0)} u\geqslant \mu^- +\frac{\omega}{C+4} -\theta(\rho)\geqslant \mu^-+\omega \biggr(\frac{1}{C+4}-\frac{1}{C_*}\biggr)=\mu^- +\frac{\omega}{2(C+4)},
\end{equation*}
provided that $C_*$ is chosen to satisfy $C_*= 2(C+4)$. This inequality yields
\begin{equation*}
\osc\limits_{B^{(\omega)}_{\frac{\rho}{2}}(x_0)} u\leqslant \biggr(1-\frac{1}{2(C+4)}\biggr)\,\omega :=\beta\,\omega,
\end{equation*}
which completes the proof of Proposition \ref{pr3.1}.

\section{Continuity of Solutions, Proof of Theorem \ref{th1.1}}
The proof of the continuity is almost standard (see \cite[Section~3]{DiB}) and uses Proposition \ref{pr3.1}. Fix $x_0\in \Omega$, let $R>0$ and construct the cube $B_{8R}(x_0)\subset \Omega$. If for all $\rho\in (0, R]$ there holds
\begin{equation}\label{eq4.1}
\osc\limits_{B_\rho(x_0)}\,u\leqslant \rho,
\end{equation}
then the continuity is evident, since $\osc\limits_{B_\rho(x_0)}\,u\to0$ as $\rho\to0$. 

Therefore, we assume that there exists $\rho_0\in (0, R]$ such that
\begin{equation}\label{eq4.2}
\osc\limits_{B_{\rho_0}(x_0)}\,u\geqslant \rho_0.
\end{equation}
In this case, we set
$$\mu^+_0:=\sup\limits_{B_{\rho_0}(x_0)} u,\quad \mu^-_0:=\inf\limits_{B_{\rho_0}(x_0)} u,\quad \omega_0:=\mu^+_0-\mu^-_0=\osc\limits_{B_{\rho_0}(x_0)} u.$$
By \eqref{eq4.2}, we conclude
$$B^{(\omega_0)}_{\rho_0}(x_0) \subset B_{\rho_0}(x_0),\quad \text{and}\quad \osc\limits_{B^{(\omega_0)}_{\rho_0}(x_0)} u\leqslant \omega_0.$$
Proposition \ref{pr3.1} yields
\begin{equation*}
\osc\limits_{B^{(\omega_0)}_{\frac{\rho_0}{2}}(x_0)} u\leqslant \max\big(\beta\,\omega_0, C_*\,\theta(\rho_0)\big)=:\omega_1,
\end{equation*}
with some $\beta \in (0,1)$ and $C_*>1$ depending only on the data. Choosing $\eta \in (0,1)$ such that
$$\frac{\eta^{\frac{p_N}{p_i}}}{\beta^{\frac{p_N-p_i}{p_i}}}\leqslant 2^{-\frac{p_N}{p_i}},\quad i=1, ..., N,\quad\text{i.e.}\quad
\eta=\frac{1}{2}\,\beta^{\frac{p_N-p_1}{p_N}},$$
we obtain
\begin{equation*}
\begin{cases}
B^{(\omega_1)}_{\rho_1}(x_0) \subset B^{(\omega_0)}_{\rho_0}(x_0),\quad \rho_1:=\eta \,\rho_0,\\
\text{and}\\
\osc\limits_{B^{(\omega_1)}_{\rho_1}(x_0)} u\leqslant \omega_1.
\end{cases}
\end{equation*}
Repeating this procedure we define for any $n\geqslant 0$
\begin{equation}\label{eq4.3}
\begin{cases}
\rho_{n+1}:=\eta^{n+1}\,\rho_0,\quad \omega_{n+1}:=\max\big(\beta\,\omega_{n}, C_*\,\theta(\rho_{n})\big),\\
\text{and}\\
\osc\limits_{B^{(\omega_{n+1})}_{\rho_{n+1}}(x_0)} u\leqslant \omega_{n+1},\quad B^{(\omega_{n+1})}_{\rho_{n+1}}(x_0) \subset B^{(\omega_n)}_{\rho_n}(x_0).
\end{cases}
\end{equation}
Note that for any $n\geqslant 0$, numbers $\omega_n$ are uniformly bounded. Indeed, since the function $\theta(\cdot)$ is non-decreasing, and $\beta<1$, we have
\begin{equation}\label{eq4.4}
\omega_n\leqslant \beta^{n}\,\omega_0+C_*\sum\limits_{i=0}^{n-1}\beta^{n-1-i}\theta(\rho_i)\leqslant \frac{1}{1-\beta}\Big(\omega_0+ C_*\,\theta(\rho_0)\Big),\quad\forall\,n\geqslant0.
\end{equation}
For any $0<r<\bar{\rho}<\rho_0$, define the numbers $j$, $m\in \mathbb{N}$ by the conditions
$$\frac{\rho_0}{2^{j}}\leqslant \bar{\rho}\leqslant \frac{\rho_0}{2^{j-1}},\quad \text{and}\quad \frac{\bar{\rho}}{2^{m}}\leqslant r\leqslant \frac{\bar{\rho}}{2^{m-1}}.$$
Now, we iterate \eqref{eq4.3} $m$ times, starting from step $j$, and apply \eqref{eq4.4} to obtain
\begin{equation}\label{eq4.5}
\begin{aligned}
\osc\limits_{B^{(\omega_{m+j})}_{\rho_{m+j}}(x_0)}\,u&\leqslant \beta^m\,\omega_j+\theta(\rho_j)\sum\limits_{i=0}^m\beta^i\leqslant
\beta^m\,\omega_j+ \frac{1}{1-\beta}\,\theta(\rho_j)
\\&\leqslant \gamma \Big(\frac{r}{\bar{\rho}}\Big)^{\alpha}\,\omega_j+\frac{1}{1-\beta}\,\theta(\bar{\rho})\leqslant \gamma \Big(\frac{r}{\bar{\rho}}\Big)^{\alpha}\,\Big[ \omega_0+ C_*\,\theta(\rho_0)\Big]+
\gamma \theta(\bar{\rho}),
\end{aligned}
\end{equation}
with $\alpha \in(0,1)$ depending only on $\beta$. For a fixed constant $\mu\in (0, 1)$, we define $\bar{\rho}$ as $\bar{\rho}:=r^{\mu}\,\rho^{1-\mu}_0$. With this choice, inequality \eqref{eq4.5} transforms into
\begin{equation}\label{eq4.6}
\osc\limits_{B^{(\nu_0)}_r(x_0)}\,u\leqslant \gamma \Big(\frac{r}{\rho_0}\Big)^{\alpha (1-\mu)}\,\Big[ \omega_0+ C_*\,\theta(\rho_0)\Big]+ \gamma \theta(r^\mu \rho^{1-\mu}_0),
\end{equation}
with $\nu_0=\frac{1}{1-\beta}\big[\omega_0+C_*\theta(\rho_0)\big]$. Since $\lim\limits_{r\rightarrow 0} \theta(r)=0$, it follows from \eqref{eq4.6} that $\osc\limits_{B^{(\nu_0)}_r(x_0)}\,u\to0$ as $r\to0$, ensuring the continuity of the solution $u$. This completes the proof of Theorem \ref{th1.1}.

%\vskip3.5mm
%{\bf Acknowledgements.}  
%The first author is supported by a grant of the National Academy of Sciences of Ukraine (project numbers is 0125U001647).

\newpage

CONTACT INFORMATION

\medskip

\medskip
\textbf{Mariia O.~Savchenko}\\Institute of Applied Mathematics and Mechanics,
National Academy of Sciences of Ukraine, \\ \indent Batiouk Str. 19, 84116 Sloviansk, Ukraine\\
Vasyl' Stus Donetsk National University,
\\ \indent 600-richcha Str. 21, 21021 Vinnytsia, Ukraine
\\shan\underline{ }maria@ukr.net

\medskip
\textbf{Igor I.~Skrypnik}\\Institute of Applied Mathematics and Mechanics,
National Academy of Sciences of Ukraine, \\ \indent Batiouk Str. 19, 84116 Sloviansk, Ukraine\\
Vasyl' Stus Donetsk National University,
\\ \indent 600-richcha Str. 21, 21021 Vinnytsia, Ukraine\\ihor.skrypnik@gmail.com

\medskip
\textbf{Yevgeniia A. Yevgenieva}
\\ Max Planck Institute for Dynamics of Complex Technical Systems, \\ \indent Sandtorstrasse 1, 39106 Magdeburg, Germany
\\Institute of Applied Mathematics and Mechanics,
National Academy of Sciences of Ukraine, \\ \indent Batiouk Str. 19, 84116 Sloviansk, Ukraine\\yevgenieva@mpi-magdeburg.mpg.de

\end{document}